\documentclass[11pt,a4paper]{amsart}

\usepackage[all,2cell]{xy}

\newtheorem{theorem}[subsection]{Theorem}
\newtheorem*{theorem*}{Theorem}

\newtheorem{lemma}[subsection]{Lemma}
\newtheorem{proposition}[subsection]{Proposition}
\newtheorem{corollary}[subsection]{Corollary}

\newtheorem*{conjecture*}{Conjecture}

\newtheorem{remark}[subsection]{Remark}

\newcommand{\ie}{{\em i.e.}\ }
\newcommand{\confer}{{\em cf.}\ }

\newcommand{\opname}[1]{\operatorname{\mathsf{#1}}}

\renewcommand{\mod}{\opname{mod}\nolimits}

\newcommand{\Mod}{\opname{Mod}\nolimits}

\newcommand{\per}{\opname{per}\nolimits}

\newcommand{\Add}{\opname{Add}\nolimits}

\newcommand{\thick}{\opname{thick}\nolimits}

\newcommand{\Hom}{\opname{Hom}}

\newcommand{\cHom}{\mathcal{H}\it{om}}

\newcommand{\ten}{\otimes}
\newcommand{\lten}{\overset{\boldmath{L}}{\ten}}

\newcommand{\cc}{{\mathcal C}}
\newcommand{\cd}{{\mathcal D}}

\newcommand{\cs}{{\mathcal S}}

\newcommand{\cx}{{\mathcal X}}

\begin{document}

\title[Tilting complexes and simple objects]{On tilting complexes providing
derived equivalences that send simple-minded objects to simple objects}

\author{Steffen Koenig}
\address{Steffen Koenig\\Universit\"{a}t Stuttgart\\
Institut f\"ur Algebra und Zahlentheorie\\ Pfaffenwaldring 57\\ D-70569 Stuttgart\\
Germany} \email{skoenig@mathematik.uni-stuttgart.de}

\author{Dong Yang}
\address{
Dong Yang\\Max-Planck-Institut f\"ur Mathematik in Bonn\\Vivatsgasse
7\\53111 Bonn\\Germany} \email{yangdong98@mails.thu.edu.cn}

\date{Last modified on \today}

\begin{abstract}
Given a set of 'simple-minded' objects in a derived category,
Rickard constructed a complex, which over a symmetric algebra
provides a derived equivalence sending the 'simple-minded' objects
to simple ones. We characterise in terms of t-structures, when this
complex is a tilting complex, show that there is an associated
natural $t$-structure and we provide an alternative construction of
this complex in terms of $A_{\infty}$-structures. Our approach is
similar to that of Keller--Nicol\'as.
\end{abstract}

\maketitle

\tableofcontents

\section{Introduction}

The module category of a finite dimensional algebra, when seen as an
abelian category, has two natural 'generators': a projective
generator and the direct sum of a full set of simple modules.
Equivalences of abelian categories send progenerators to
progenerators and simples to simples. The derived module category of
an algebra, when seen as a triangulated category, has two kinds of
natural generators: each tilting complex 'generates' the category,
and the direct sum of a full set of simple modules does so, too.
Equivalences of derived categories send tilting complexes to tilting
complexes. It is, however, not clear what happens to simple modules
under derived equivalences. For symmetric algebras, Rickard
\cite{Rickard02} has shown that the 'group' of derived equivalences
acts transitively on the class of objects sharing with the simple
objects certain obvious conditions. Given such 'simple-minded'
objects, he explicitly constructed a tilting complex, and thus a
derived equivalence. This has been used extensively in modular
representation theory of finite groups. Rickard's construction,
using Milnor colimits, produces a complex for any algebra, not just
a symmetric one. To show that this complex is a tilting complex, the
assumption symmetric is used, and in general one cannot expect to
get a tilting complex.

This note addresses two questions in this context. First, we
characterise in terms of t-structures
(Section~\ref{s:construction-rickard}), when Rickard's construction
yields a tilting complex. On the way, we give new proofs of some
results by Rickard and by Al-Nofayee
\cite{Al-Nofayee07,Al-Nofayee09}, who also considered this problem
and obtained related results, in particular also extending Rickard's
main result to self-injective algebras, and constructing a
$t$-structure. Similar results are obtained by Keller and Nicol\'as
in~\cite{KellerNicolas10} in a different context. Secondly, we
provide in Section~\ref{s:alternative-construction} an alternative
construction of the same complex, in terms of
$A_{\infty}$-categories. This uses work of Keller and Lef\`evre
\cite{Lefevre03}. In an appendix we investigate some basic
properties of non-positively graded finite-dimensional dg algebras.
These properties are used in Section~\ref{s:construction-rickard} to construct the $t$-structure and to extend Rickard's result to self-injective algebras, and used in Section~\ref{s:alternative-construction} to show that the above
results are valid also in the more general setting of finite-dimensional
non-positive dg algebras.

\section*{Acknowledgements}
The second named author gratefully acknowledges financial support from
Max-Planck-Institut f\"ur Mathematik in Bonn. He is deeply grateful
to Bernhard Keller for valuable conversations on derived categories
of dg algebras and $A_{\infty}$-algebras.

\section{Notations and preliminaries} Throughout, $K$ will be an
algebraically closed field. All algebras, modules, vector spaces and
categories are over the base field $K$. For a category $\cc$, we
denote by $\Hom_{\cc}(X,Y)$ the morphism space from $X$ to $Y$,
where $X$ and $Y$ are two objects of $\cc$. We will omit the
subscript and write $\Hom(X,Y)$ when it does not cause confusion. By
abuse of notation, we will denote by $\Sigma$ the suspension
functors of all the triangulated categories appearing in this paper. 
For a triangulated category $\cc$ and a set $\cs$ of objects in $\cc$, let
$\thick(\cs)$ denote the smallest triangulated subcategory of $\cc$ containing 
objects in $\cs$ and stable for taking direct summands, and let $\Add(\cs)$ denote the smallest
full subcategory of $\cc$ containing all objects of $\cs$ and stable
for taking coproducts and direct summands.

For a finite-dimensional basic algebra $\Lambda$, let $\Mod\Lambda$
(respectively, $\mod\Lambda$) denote the category of right
$\Lambda$-modules (respectively, finite-dimensional right
$\Lambda$-modules), and let $\cd(\Mod\Lambda)$ (respectively,
$\cd^b(\mod\Lambda)$, $\cd^{-}(\mod\Lambda)$) denote the derived
category of $\Mod\Lambda$ (respectively, bounded derived category of
$\mod\Lambda$, bounded above derived category of $\mod\Lambda$). The
categories $\cd(\Mod(\Lambda))$, $\cd^{-}(\mod\Lambda)$ and
$\cd^b(\mod\Lambda)$ are triangulated with suspension functor the
shift functor. We view $\cd^{-}(\mod\Lambda)$ and
$\cd^b(\mod\Lambda)$ as triangulated subcategories of
$\cd(\Mod(\Lambda))$.

For a differential graded(=dg) algebra $A$, let $\cd(A)$ denote the
derived category of right dg $A$-modules, \confer~\cite{Keller94},
and let $\cd_{fd}(A)$ denote its full subcategory of dg $A$-modules
whose total cohomology is finite-dimensional. They are triangulated
categories with suspension functor the shift functor. Let $\per(A)=\thick(A_A)$, \ie the smallest triangulated
subcategory of $\cd(A)$ containing the free dg $A$-module of rank
$1$ and stable for taking direct summands. Let $A$ and $B$ be two dg
algebras. Then a triangle equivalence between $\cd(A)$ and $\cd(B)$
restricts to a triangle equivalence between $\per(A)$ and $\per(B)$
as well as a triangle equivalence between $\cd_{fd}(A)$ and
$\cd_{fd}(B)$. If $A$ is a finite-dimensional algebra viewed as a dg
algebra concentrated in degree 0, then $\cd(A)$ is exactly $\cd(\Mod
A)$, $\cd_{fd}(A)$ is $\cd^b(\mod A)$, and $\per(A)$ is triangle
equivalent to the homotopy category of bounded complexes of finitely
generated projective $A$-modules.

\section{Rickard's construction}\label{s:construction-rickard}

Let $\Lambda$ be a finite-dimensional basic $K$-algebra. In this
section we discuss a construction by Rickard~\cite{Rickard02}. The
same construction is studied by
Keller--Nicol\'as~\cite{KellerNicolas10} in the context of positive
dg algebras.

Let $r$ be the rank of the Grothendieck group of $\mod\Lambda$.
Following~\cite{KoenigLiu10} we say that a set of objects
$X_1,\ldots,X_r$ in the bounded derived category
$\mathcal{D}^{b}(\mod\Lambda)$ are \emph{simple-minded} if the
following conditions hold
\begin{itemize}
\item[(1)] $\Hom(X_i,\Sigma^m X_j)=0,~~\forall~m<0$,
\item[(2)] $\Hom(X_i,X_j)=\begin{cases} K & \text{if\ }i=j,\\
                                           0 & \text{otherwise},
                                           \end{cases}$
\item[(3)] $X_1,\ldots,X_r$ generates
$\mathcal{D}^{b}(\mod\Lambda)$, \ie~$\cd^b(\mod\Lambda)=\thick(X_1,\ldots,X_r)$.
\end{itemize}

In~\cite{Rickard02} Rickard constructed from $X_1,\ldots,X_r$ a set
of objects $T_1,\ldots,T_r$ as follows.

Set $X_i^{(0)}=X_i$. Suppose $X_i^{(n-1)}$ is constructed. For
$i,j=1,\ldots,r$ and $m<0$, let $B(j,m,i)$ be a basis of
$\Hom(\Sigma^m X_j,X_i^{(n-1)})$. Put
\[Z_i^{(n-1)}=\bigoplus_{m<0}\bigoplus_j \bigoplus_{B(j,m,i)}\Sigma^m
X_{j}\] and let $\alpha_i^{(n-1)}:Z_i^{(n-1)}\rightarrow
X_i^{(n-1)}$ be the map whose component corresponding to $f\in
B(j,m,i)$ is exactly $f$.

Let $X_i^{(n)}$ be a cone of $\alpha_i^{(n-1)}$ and form the
corresponding triangle
\[\xymatrix{X_i^{(n-1)}\ar[r]^{\alpha_i^{(n-1)}}
& X_i^{(n-1)}\ar[r]^{\beta_i^{(n-1)}}& X_i^{(n)}\ar[r]& \Sigma
Z_i^{(n-1)}.}\] Inductively we obtain a sequence of morphisms in
$\cd(\Mod\Lambda)$:
\[\xymatrix{X_i^{(0)}\ar[r]^{\beta_i^{(0)}} & X_i^{(1)}\ar[r] &\ldots\ar[r] &
X_i^{(n-1)}\ar[r]^{\beta_i^{(n-1)}} & X_i^{(n)}\ar[r] & \ldots.}\]
Let $T_i$ be the Milnor colimit of this sequence. Up to isomorphism,
$T_i$ is determined by the following triangle
\[\xymatrix{\bigoplus_{n\geq 0}X_i^{(n)} \ar[r]^{id-\beta} & \bigoplus_{n\geq 0}X_i^{(n)}\ar[r] & T_i\ar[r] &\Sigma \bigoplus_{n\geq 0}X_i^{(n)}}.\]

\subsection{Properties of $T_i$'s}\label{ss:properties}

The following properties of $T_i$'s were proved in~\cite{Rickard02}
for symmetric algebras $\Lambda$, but the proofs there also work in
general.

\begin{lemma}\label{L:T}
\begin{itemize}
\item[a)] (\cite[Lemma 5.4]{Rickard02}) For $1\leq i,j\leq r$, and
$m\in\mathbb{Z}$,
\[\Hom(X_j,\Sigma^m T_{i})=\begin{cases} K & \text{if }i=j\text{ and } m=0,\\
                                           0 & \text{otherwise}.
                                           \end{cases}\]

\item[b)] (\cite[Lemma 5.5]{Rickard02}) For each $1\leq i\leq r$, $T_i$ is
isomorphic to a bounded complex of finitely generated injectives.
>From now on we assume that $T_i$ is such a complex.

\item[c)] (\cite[Lemma 5.8]{Rickard02}) Let $C$ be an object of
$\mathcal{D}^{-}(\mod \Lambda)$. If $\Hom(C,\Sigma^m T_i)=0$ for all
$m\in\mathbb{Z}$ and all $1\leq i\leq r$, then $C=0$.
\end{itemize}
\end{lemma}

Let $\nu$ be the Nakayama functor, and $\nu^{-1}$ the inverse
Nakayama functor ({\em cf.}~\cite[Chapter 1, Section
4.6]{Happel88}). They are quasi-inverse triangle equivalences
between the triangulated subcategories $\per(\Lambda)$ and
$\thick(\mathrm{D}(_{\Lambda}\Lambda))$ of $\cd(\Mod\Lambda)$, where
$\mathrm{D}=\Hom_{K}(?,K)$ is the duality functor. The following is
a consequence of Lemma~\ref{L:T} and the property of the Nakayama
functor.

\begin{lemma}\label{L:nuinverseT}
\begin{itemize}
\item[a)] For $1\leq i,j\leq r$, and $m\in\mathbb{Z}$,
\[\Hom(\nu^{-1}T_{i},\Sigma^m X_j)=\begin{cases} K & \text{if }i=j\text{ and } m=0,\\
                                           0 & \text{otherwise}.
                                           \end{cases}\]
\item[b)] For each $1\leq i\leq r$, $\nu^{-1}T_i$ is a
bounded complex of finitely generated projectives.
\item[c)] Let $C$
be an object of $\mathcal{D}^{-}(\mod \Lambda)$. If
$\Hom(\nu^{-1}T_i,\Sigma^m C)=0$ for all $m\in\mathbb{Z}$ and all
$1\leq i\leq r$, then $C=0$.
\end{itemize}
\end{lemma}

We put $T=\bigoplus_{i=1}^{r} T_i$ and
$\nu^{-1}T=\bigoplus_{i=1}^{r}\nu^{-1}T_i$.

\begin{lemma}\label{L:vanishingofposext}
We have \[\Hom(\nu^{-1}T,\Sigma^m T)=0\] for $m<0$. Equivalently,
\[\Hom(\nu^{-1}T,\Sigma^m \nu^{-1}T)=\Hom(T,\Sigma^m T)=0\] for $m>0$.
\end{lemma}
\begin{proof} Same as the proof of~\cite[Lemma 5.2]{Rickard02}, with the
$T_i$ in the first entry of $\Hom$ there replaced by $\nu^{-1}T_i$.
\end{proof}

\begin{theorem}[\cite{Rickard02} Theorem 5.1] \label{T:case-symmetric} When $\Lambda$ is a symmetric algebra, $T=\nu^{-1}T$ is
a tilting complex.
\end{theorem}
\begin{proof} By Lemma~\ref{L:nuinverseT}, $\nu^{-1}T$ is a compact
generator of $\mathcal{D}(\Mod\Lambda)$. Moreover, when $\Lambda$ is
symmetric, the Nakajama functor is isomorphic to the identity. The
desired result follows from Lemma~\ref{L:vanishingofposext}.
\end{proof}

In general, we may ask when $\nu^{-1}T$ is a tilting complex. If
this is the case, then by Rickard's Morita's theorem for derived
categories (\confer~\cite{Rickard89}) we have a triangle equivalence
\[\cd(\Mod\Lambda)\simeq \cd(\Mod\Gamma),\]
which takes $X_1,\ldots,X_r$ to a complete set of non-isomorphic
simple $\Gamma$-modules, where $\Gamma=\Hom(\nu^{-1}T,\nu^{-1}T)$.
Conversely, assume there is a finite-dimensional algebra $\Gamma$
with an equivalence $F:\cd(\Mod\Lambda)\simeq \cd(\Mod\Gamma)$
sending $X_1,\ldots,X_r$ to a complete set of non-isomorphic simple
$\Gamma$-modules. It follows from Lemma~\ref{L:nuinverseT} that for
$1\leq i,j\leq r$ and $m\in\mathbb{Z}$, we have
\[\Hom_{\mathcal{D}(\Mod\Gamma)}(F\nu^{-1}T_{i},\Sigma^m FX_j)=\begin{cases} K & \text{if }i=j\text{ and } m=0,\\
                                           0 & \text{otherwise}.
                                           \end{cases}\]
This means that $F\nu^{-1}T_i$ is the projective cover of $FX_i$,
and hence $F\nu^{-1}T=\Gamma$ is the free $\Gamma$-module of rank
$1$. Thus $\nu^{-1}T$ is a tilting complex.

\subsection{A $t$-structure}\label{ss:t-structure}

Let $\cc$ be a triangulated category. A \emph{$t$-structure} on
$\cc$ (\cite{BeilinsonBernsteinDeligne82}) is a pair $(\cc^{\leq
0},\cc^{\geq 0})$ of strictly full subcategories of $\cc$ such that
\begin{itemize}
\item[$\cdot$] $\Sigma\cc^{\leq 0}\subseteq\cc^{\leq 0}$ and
$\Sigma^{-1}\cc^{\geq 0}\subseteq\cc^{\geq 0}$;
\item[$\cdot$] $\Hom_{\cc}(M,\Sigma^{-1}N)=0$ for $M\in\cc^{\leq 0}$
and $N\in\cc^{\geq 0}$,
\item[$\cdot$] for each $M\in\cc$ there is a triangle $M'\rightarrow
M\rightarrow M''\rightarrow\Sigma M'$ in $\cc$ with $M'\in\cc^{\leq
0}$ and $M''\in\Sigma^{-1}\cc^{\geq 0}$.
\end{itemize}
The \emph{heart} $\cc^{\leq 0}~\cap~\cc^{\geq 0}$ is always abelian.
The $t$-structure $(\cc^{\leq 0},\cc^{\geq 0})$ is said to be
\emph{bounded} if $$\bigcup_{n\in\mathbb{Z}}\Sigma^n \cc^{\leq
0}=\cc=\bigcup_{n\in\mathbb{Z}}\Sigma^n\cc^{\geq 0}.$$ A typical
example of a $t$-structure is the pair $(\cd^{\leq 0},\cd^{\geq 0})$
for the derived category $\cd=\cd(\Mod A)$ of an (ordinary) algebra
$A$, where $\cd^{\leq 0}$ consists of complexes with vanishing
cohomologies in positive degrees, and $\cd^{\geq 0}$ consists of
complexes with vanishing cohomologies in negative degrees. This
$t$-structure restricts to a bounded $t$-structure of $\cd^b(\mod
A)$.

Assume $\Lambda$, $X_1,\ldots,X_r$, $T$ as in the preceding
subsection. Recall that by Lemma~\ref{L:nuinverseT}, $\nu^{-1}T$ is
a compact generator of $\mathcal{D}(\Mod\Lambda)$. The following
proposition is an immediate consequence of
Lemma~\ref{L:vanishingofposext} and the definition of the
$t$-structure.

\begin{proposition}\label{P:criterion}
The following assertions are equivalent:
\begin{itemize}
\item[(i)] $\nu^{-1}T$ is a tilting complex,
\item[(ii)] $\nu^{-1}T$ is in the heart of some $t$-structure.
\end{itemize}
\end{proposition}

There are two natural $t$-structures related to the set
$X_1,\ldots,X_r$. Let $X^{\leq 0}$ be the smallest full subcategory
of $\mathcal{D}(\Mod\Lambda)$ containing $X_1,\ldots,X_r$ and stable
for taking suspensions, extensions and coproducts.
By~\cite[Proposition 3.2]{AlonsoJeremiasSouto03}, the pair $(X^{\leq
0},\Sigma(X^{\leq 0})^{\perp})$ is a $t$-structure of
$\mathcal{D}(\Mod\Lambda)$, where $(X^{\leq 0})^{\perp}$ is the full
subcategory of $\mathcal{D}(\Mod\Lambda)$ of objects $M$ such that
$\Hom(N,M)=0$ for any object $N$ of $X^{\leq 0}$. Dually so is
$(\Sigma^{-1}{}^{\perp}\hspace{-1pt}(X^{\geq 0}),X^{\geq 0})$, where
$X^{\geq 0}$ is defined as the smallest full subcategory of
$\mathcal{D}(\Mod\Lambda)$ which contains $X_1,\ldots,X_r$ and which
is stable for taking cosuspensions, extensions and products, and
${}^{\perp}\hspace{-1pt}(X^{\geq 0})$  is the full subcategory of
$\mathcal{D}(\Mod\Lambda)$ of objects $M$ such that $\Hom(M,N)=0$
for any object $N$ of $X^{\geq 0}$.

Yet there is a third natural $t$-structure. Let $\tilde\Gamma$ be
the dg endomorphism algebra of $\nu^{-1}T$. Precisely, the degree
$n$ component of $\tilde{\Gamma}$ consists of those $\Lambda$-linear
maps from $\nu^{-1}T$ to itself which are homogeneous of degree $n$,
and the differential of $\tilde{\Gamma}$ takes a homogeneous map $f$
of degree $n$ to $d\circ f-(-1)^n f\circ d$, where $d$ is the
differential of $\nu^{-1}T$. We have
$H^m(\tilde{\Gamma})=\Hom(\nu^{-1}T,\Sigma^m \nu^{-1}T)$ for any
integer $m$. The dg algebra $\tilde{\Gamma}$ is finite-dimensional
by Lemma~\ref{L:nuinverseT} b), and it has cohomology concentrated
in non-positive degrees by Lemma~\ref{L:vanishingofposext}. It
follows that the derived category $\cd(\tilde\Gamma)$ carries a
natural $t$-structure $(\cd^{\leq 0}, \cd^{\geq 0})$, where
$\cd^{\leq 0}$ is the full subcategory of $\cd(\tilde\Gamma)$
consisting of dg $\tilde{\Gamma}$-modules $M$ with $H^{m}(M)=0$ for
$m>0$, and $\cd^{\geq 0}$ is the full subcategory of
$\cd(\tilde\Gamma)$ consisting of dg $\tilde\Gamma$-modules $M$ with
$H^{m}(M)=0$ for $m<0$, and the heart $\cd^{\leq 0}\cap\cd^{\geq 0}$
is equivalent to $\Mod\Gamma$, where $\Gamma=H^{0}(\tilde\Gamma)$,
see the appendix. This $t$-structure restricts to a $t$-structure of
$\cd_{fd}(\tilde{\Gamma})$, denoted by $(\cd^{\leq 0}_{fd},\cd^{\geq
0}_{fd})$, whose heart is equivalent to $\mod\Gamma$.

The complex $\nu^{-1}T$ has a natural dg
$\tilde{\Gamma}$-$\Lambda$-bimodule structure. By~\cite[Lemma 6.1
(a)]{Keller94}, we have a triangle equivalence
\[?\lten_{\tilde\Gamma}\nu^{-1}T:\mathcal{D}(\tilde\Gamma)\stackrel{\sim}{\longrightarrow}\mathcal{D}(\Mod\Lambda).\]
This equivalence takes $\tilde\Gamma$ to $\nu^{-1}T$, takes a
complete set of non-isomorphic simple $\Gamma$-modules to
$X_1,\ldots,X_r$, and restricts to a triangle equivalence between
$\cd_{fd}(\tilde{\Gamma})$ and
$\cd_{fd}(\Lambda)=\cd^b(\mod\Lambda)$. The image of the
$t$-structure $(\cd^{\leq 0},\cd^{\geq 0})$ under the triangle
equivalence $?\lten_{\tilde\Gamma}\nu^{-1}T$ is a $t$-structure of
$\cd(\Mod\Lambda)$, which we still denote by $(\cd^{\leq
0},\cd^{\geq 0})$. The image of the $t$-structure $(\cd^{\leq
0}_{fd},\cd^{\geq 0}_{fd})$ is exactly the $t$-structure $(\cc^{\leq
0},\cc^{\geq 0})$ in~\cite{Al-Nofayee09}.

\begin{proposition}\label{p:t-structures} The above three $t$-structures $(X^{\leq 0},\Sigma(X^{\leq
0})^{\perp})$, $(\Sigma^{-1}{}^{\perp}\hspace{-1pt}(X^{\geq
0}),X^{\geq 0})$ and $(\cd^{\leq 0},\cd^{\geq 0})$ coincide.
\end{proposition}
\begin{proof} If suffices to prove that $X^{\leq 0}=\cd^{\leq 0}$
and $X^{\geq 0}=\cd^{\geq 0}$. We only prove the first statement,
and the second statement is dual. Let $Y^{\leq 0}$ be the image of
$X^{\leq 0}$ under a quasi-inverse of
$?\lten_{\tilde\Gamma}\nu^{-1}T$, \ie $Y^{\leq 0}$ is the smallest
full subcategory of $\cd(\tilde{\Gamma})$ containing the simple
$\Gamma$-modules and stable for taking suspensions, extensions and
coproducts. We shall prove the equivalent statement $Y^{\leq
0}=\cd^{\leq 0}$ in $\cd(\tilde\Gamma)$.

Let $M$ be a dg $\tilde\Gamma$-module whose cohomology is
concentrated in non-positive degrees.  Then the graded module
$H^*(M)$ over the non-positively graded algebra $H^*(\tilde\Gamma)$
admits an $\Add(\Sigma^m H^*(\tilde\Gamma)|m\geq 0)$-resolution. It
follows from~\cite[Theorem 3.1 (c)]{Keller94} that $M$ belongs to
the the smallest full subcategory of $\cd(\tilde\Gamma)$ containing
$\tilde\Gamma$ and stable for taking suspensions, extensions and
coproducts. Therefore, this latter category coincides with
$\cd^{\leq 0}$. But it is clear that $Y^{\leq 0}$ contains
$\tilde\Gamma$, and hence $Y^{\leq 0}$ contains $\cd^{\leq 0}$. The
inclusion in the other direction is obvious.
\end{proof}

An abelian category is a \emph{length category} if every object in
it has finite length. Two sets of simple-minded objects are
\emph{equivalent} if they have the same closure under extensions.
The following is a counterpart of~\cite[Corollary
11.5]{KellerNicolas10}.

\begin{corollary} There is a bijection between the set of bounded
$t$-structures of $\cd^b(\mod\Lambda)$ whose heart is a length
category with finite many simple objects (up to isomorphism) and the
set of equivalence classes of families of simple-minded objects of
$\cd^b(\mod\Lambda)$. In particular, the heart of a bounded
$t$-structure of $\cd^b(\mod\Lambda)$ is a length category if and
only if it is equivalent to $\mod\Gamma$ for some finite-dimensional
algebra $\Gamma$.
\end{corollary}

We remind the reader that the heart of a $t$-structure of
$\cd^b(\mod\Lambda)$ is not always a length category. For example,
the derived category of the path algebra of the Kronecker quiver has
a $t$-structure whose heart is the category of coherent sheaves over
the projective line, which is not a length category.

\subsection{The case of self-injective algebras}\label{ss:injective-algebra}
Al-Nofayee in~\cite{Al-Nofayee07} extended Rickard's result
Theorem~\ref{T:case-symmetric} to the case when $\Lambda$ is a self-injective
algebra; then $T=\nu^{-1}T$ is a tilting complex. This result can now be
derived again.

Let $\Lambda$ be a finite-dimensional self-injective algebra. In
this case, the two categories $\per(\Lambda)$ and
$\thick(\mathrm{D}(_{\Lambda}\Lambda))$ coincide. The Nakayama
functor $\nu$ and its quasi-inverse $\nu^{-1}$ can be extended to
auto-equivalences of $\cd^{-}(\mod\Lambda)$ because each object in
$\cd^{-}(\mod\Lambda)$ admits a projective resolution which is
bounded above and whose components are finite generated.

Let $X_1,\ldots,X_r$ be a set of simple-minded objects in
$\mathcal{D}^{b}(\mod\Lambda)$ stable under the Nakayama functor
$\nu$. Let $T$ be constructed as in
Section~\ref{s:construction-rickard}.
\begin{proposition}[\cite{Al-Nofayee07} Theorem 4] The complex $T$ is a tilting complex.
\end{proposition}
\begin{proof}
The Nakayama functor $\nu$ induces a permutation on the set
$\{1,\ldots,r\}$, also denoted by $\nu$, given by
$X_{\nu(i)}=\nu(X_i)$ for $i=1,\ldots,r$. Applying $\nu$ to the
formula Lemma~\ref{L:nuinverseT} a), we obtain for $1\leq i,j\leq
r$, and $m\in\mathbb{Z}$,
\[\Hom(T_{i},\Sigma^m X_{\nu(j)})=\begin{cases} K & \text{if }i=j\text{ and } m=0,\\
                                           0 & \text{otherwise}.
                                           \end{cases}\]
Applying a quasi-inverse $G$ of the triangle equivalence
$?\lten_{\tilde{\Gamma}}\nu^{-1}T:\cd(\tilde{\Gamma})\rightarrow
\cd(\Lambda)$, we obtain
\[\Hom(GT_{i},\Sigma^m GX_{\nu(j)})=\begin{cases} K & \text{if }i=j\text{ and } m=0,\\
                                           0 & \text{otherwise}.
                                           \end{cases}\]
Since $GX_{\nu(1)},\ldots,GX_{\nu(r)}$ is a complete set of simple
$\tilde{\Gamma}$-modules, it follows from Section~\ref{S:simpledg} (the appendix) that $GT_1,\ldots,GT_r$ sum up
to the free module $\tilde{\Gamma}$. Recall that $G(\nu^{-1}T)\cong\tilde\Gamma$. As a
consequence, we have $T\cong\nu^{-1}T$. Now the desired result follows
from Lemma~\ref{L:vanishingofposext}.
\end{proof}

\section{An alternative
construction}\label{s:alternative-construction}

In this section we will give another construction of $T$ using the
$A_\infty$-version of Morita's theorem for triangulated categories
(\confer~\cite{Lefevre03}). Let us first recall the definition and
basic properties of $A_\infty$-algebras and $A_\infty$-modules.

\subsection{$A_\infty$-algebras and $A_\infty$-modules}
We follow~\cite{Lefevre03}. \cite{Keller01} and
\cite{LuPalmieriWuZhang08} are also nice references.

An \emph{$A_\infty$-algebra} is a graded vector space $A$ endowed
with a family of homogeneous maps
\[m_n: A^{\ten n}\longrightarrow A, n\geq 1\]
of degree $2-n$ satisfying the equations
\[\sum_{j+k+l=n}(-1)^{jk+l}m_{j+1+l}(id^{\ten j}\ten m_k\ten id^{\ten l})=0, n\geq 1.\]
These $m_n$ are called the \emph{multiplications} of $A$. The
$A_\infty$-algebra $A$ is \emph{minimal} if $m_1=0$. We say that $A$
is \emph{strictly unital} if $A$ has a \emph{strict unit}, \ie a
homogeneous element $1_A$ of degree $0$ such that for $n\neq 2$ the
multiplication $m_n$ has value zero if one of its $n$ arguments
equals $1_A$, and
\[m_2(1_A\ten a)=a=m_2(a\ten 1_A)\]
for all $a$ in $A$. 

Let $A$ be a strictly unital $A_\infty$-algebra. A \emph{(right)
$A_\infty$-module} over $A$ is a graded vector space $M$ endowed
with a family of homogeneous maps
\[m_n^M:M\ten A^{\ten n-1}\longrightarrow M, n\geq 1\]
of degree $2-n$ such that
\[\sum_{j+k+l=n}(-1)^{jk+l}m_{j+1+l}(id^{\ten j}\ten m_k\ten id^{\ten l})=0.\]
An $A_\infty$-module $M$ \emph{minimal} if $m_1^M=0$, and is
\emph{strictly unital} if one of $a_2,\ldots,a_n$ equals $1_A$
implies
\[m_n^M(m\ten a_2\ten
\cdots\ten a_n)=0\] for all $n\geq 3$, and $m_2^M(m\ten 1_A)=m$ for
all $m$ in $M$.

Let $M$ and $M'$ be two strictly unital $A_\infty$-modules over $A$.
An \emph{$A_\infty$-morphism} $f:M\rightarrow M'$ is a family of
homogeneous maps \[f_n:M\ten A^{\ten n-1}\longrightarrow M', n\geq
1\] of degree $1-n$ satisfying the following identity for all $n\geq
1$
\[\sum(-1)^{jk+l}f_{j+1+l}(id^{\ten j}\ten m_k\ten id^{\ten l})=\sum m_{s+1}(f_r\ten id^{\ten
s}),\] where $j+k+l=n$ and $r+s=n$. In particular, $f_1$ is a chain
map of complexes. The $A_\infty$-morphism $f$ is a
\emph{quasi-isomorphism} if $f_1$ induces identities on all
cohomologies. $f$ is \emph{strictly unital} if one of
$a_2,\ldots,a_n$ equals $1_A$ implies
\[f_n(m\ten a_2\ten\cdots\ten a_n)=0\]
for all $n\geq 2$. $f$ is \emph{strict} if $f_n=0$ for all $n\geq
2$.

Let $\Mod_{\infty}(A)$ be the category of strictly unital
$A_\infty$-modules over $A$ with strictly unital
$A_\infty$-morphisms as morphisms. The \emph{derived category}
$\cd(A)$ is the category obtained from $\Mod_{\infty}(A)$ by
formally inverting all quasi-isomorphisms. The category $\cd(A)$ is
a triangulated category whose suspension functor $\Sigma$ the shift
functor. For a strictly unital $A_\infty$-module $M$ over $A$ and an
integer $i$, we have
\[\Hom_{\cd(A)}(A,\Sigma^i M)=H^iM.\]

Let $\per(A)$ denote the triangulated subcategory of $\cd(A)$
generated by the free module of rank $1$, and $\cd_{fd}(A)$ denote
the full subcategory of $\cd(A)$ consisting of those
$A_\infty$-modules whose total cohomology is finite-dimensional.

\begin{lemma} Let $A$ and $B$ be two strictly unital
$A_\infty$-algebras. A triangle equivalence $\cd(A)\rightarrow
\cd(B)$ restricts to triangle equivalences $\per(A)\rightarrow
\per(B)$ and $\cd_{fd}(A)\rightarrow \cd_{fd}(B)$.
\end{lemma}

\begin{proposition}[\cite{Lefevre03} Proposition 3.3.1.7]\label{p:minimal-model}
Let $A$ be a strictly unital $A_\infty$-algebra, and $M$ be a
strictly unital $A_\infty$-module over $A$. Then there is a strictly
unital minimal $A_\infty$-module $M'$ over $A$ together with a
quasi-isomorphism of strictly unital $A_\infty$-modules from $M'$ to
$M$.
\end{proposition}

\begin{theorem}[\cite{Lefevre03} Theorem 7.6.0.6]\label{t:morita}
Let $\cc$ be an algebraic triangulated category, \ie $\cc$ is
triangle equivalent to the stable category of a Frobenius category.
Assume $\cc$ has split idempotents and $\cc$ is generated by an
object $X$. Then there is a strictly unital minimal
$A_\infty$-algebra $A$ with strict unit $id_X$ and a triangle
equivalence
\[\cc\longrightarrow \per(A)\] taking $X$ to $A$.
It follows that the underlying graded algebra of $A$ is the graded
endomorphism algebra
$\bigoplus_{i\in\mathbb{Z}}\Hom_{\cc}(X,\Sigma^i X)$. Moreover,
$m_n(a_1\ten\cdots\ten a_n)=0$ for all $n\neq 2$ if one of $a_j$ is
the identity morphism of a direct summand of $X$.
\end{theorem}

\subsection{Minimal positive $A_\infty$-algebras}

Let $A$ be a strictly unital minimal positive $A_\infty$-algebra,
\ie $A$ is strictly unital and minimal and satisfies
\begin{itemize}
\item $A^i=0$ for all negative integers $i$,
\item $A^0$ is the product of $r$ copies of the base field $K$ for some positive integer $r$,
\item
$m_n(a_1\ten\cdots\ten a_n)=0$ if one of $a_1,\ldots,a_n$ belongs to
$A^0$.
\end{itemize}

Put $A^{>0}=\bigoplus_{i>0}A^i$. Then $A^{>0}$ is an
$A_\infty$-ideal of $A$: the multiplication $m_n$ takes value in
$A^{>0}$ if one of the $n$ arguments belongs to $A^{>0}$. Let
$\bar{A}$ denote the quotient $A_\infty$-algebra of $A/A^{>0}$. It
has vanishing higher multiplications, and is isomorphic to the
product of $r$ copies of $K$.

Let $1=e_1+\ldots+e_r$ be the unique (up to reordering)
decomposition of the identity of $A^0$ into the sum of primitive
orthogonal idempotents. Then each $P_i=e_i A$ is an
$A_\infty$-submodule of the free module of rank $1$:
\[m_n(e_i a\ten a_2\ten\cdots\ten a_n)=-(-1)^n e_i m_n(a\ten a_2\ten\cdots\ten a_n).\]
The subspace $e_i A^{>0}$ is an $A_\infty$-submodule of $P_i$, and
the quotient $A_\infty$-module $S_i=P_i/e_i A^{>0}$ is
$1$-dimensional with basis the class of $e_i$. We call
$S_1,\ldots,S_r$ \emph{simple modules} over $A$. Viewed as an
$A_\infty$-module over $A$, $\bar{A}$ is isomorphic to the direct
sum of $S_1,\ldots,S_r$. The following two lemmas are also proved in~\cite{KellerNicolas10} (in the form of dg algebras and dg modules).

\begin{lemma}\label{l:finite-dim-module} Let $A$ be a strictly unital minimal positive
$A_\infty$-algebra. Then \[\Hom_{\cd(A)}(\bar{A},\Sigma^m
\bar{A})=0\] for positive integers $m$.
\end{lemma}
\begin{proof} The graded module $\bar{A}$ over the positively graded algebra $A$
admits an $\Add(\Sigma^{-m} A|m\geq 0)$-resolution. Now the desired result follows 
from an $A_\infty$-version of~\cite[Theorem 3.1 (c)]{Keller94} 
(one can obtain this, for example,  by going to the enveloping dg algebra). \end{proof}

\begin{lemma} Let $A$ be a strictly unital minimal positive
$A_\infty$-algebra. Then the triangulated category $\cd_{fd}(A)$ is
generated by $\bar{A}$.
\end{lemma}
\begin{proof} By Proposition~\ref{p:minimal-model}, it suffices to prove that
a finite-dimensional strictly unital minimal $A_\infty$-module $M$
over $A$ is generated by $\bar{A}$. Up to shift we may assume that
$M^i=0$ for all negative integers $i$ and $M^0\neq 0$. Put
$M^{>0}=\bigoplus_{i>0}M^i$. Then $M^{>0}$ is an
$A_\infty$-submodule of $M$, and we have a triangle in $\cd(A)$
\[\xymatrix{M^{>0}\ar[r] & M\ar[r] & \bar{M}\ar[r] & \Sigma M^{>0}.}\]
Here $\bar{M}$ is the $A_\infty$-quotient module $M/M^{>0}$, and is
concentrated in degree $0$. Its structure of an $A_\infty$-module
comes from its structure of an $\bar{A}$-module, and hence is
generated by $\bar{A}$. Now by induction on the dimension of $M$ we
finish the proof.
\end{proof}

\subsection{The alternative construction}

Let $\Lambda$ be a finite-dimensional basic $K$-algebra. Let
$S_1,\ldots,S_r$ be a complete set of representatives of simple
$\Lambda$-modules.

By Theorem~\ref{t:morita}, there is a strictly unital minimal
positive $A_\infty$-algebra
$$\cs=\bigoplus_{m\in\mathbb{Z}}\Hom(\bigoplus_i
S_i,\Sigma^m\bigoplus_i S_i)$$ (the $A_{\infty}$-Koszul dual of
$\Lambda$) and a triangle equivalence
\[\xymatrix{\Phi:\cd^b(\mod\Lambda) \ar[r] & \per(\cs)}\]
taking $S_j$ ($j=1,\ldots,r$) to
$P_j=\bigoplus_{m\in\mathbb{Z}}\Hom(\bigoplus_i S_i,\Sigma^m S_j)$.

The indecomposable injective $\Lambda$-modules $I_1,\ldots,I_r$ are
characterized by the property
\[\Hom(S_i,\Sigma^m I_j)=\begin{cases} K & \text{ if } i=j \text{
and } m=0,\\
0 & \text{ otherwise.}
\end{cases}
\]
So their images $\Phi(I_1),\ldots,\Phi(I_r)$ under the equivalence
$\Phi$ are characterized by the property
\[\Hom(P_i,\Sigma^m \Phi(I_j))=\begin{cases} K & \text{ if } i=j \text{
and } m=0,\\
0 & \text{ otherwise.}
\end{cases}
\]
Therefore, $\Phi(I_1),\ldots,\Phi(I_r)$ are precisely the
indecomposable direct summands of $\bar{\cs}$. In other words, the
equivalence $\Phi$ restricts to a triangle equivalence
\[\xymatrix{\Phi|:\thick(\mathrm{D}(_{\Lambda}\Lambda)=\thick(I_1,\ldots,I_r)\ar[r] & \thick(\bar{\cs})=\cd_{fd}(\cs),}\]
where the last equality follows from
Lemma~\ref{l:finite-dim-module}.

Let $X_1,\ldots,X_r\in\mathcal{D}^{b}(\mathrm{mod}\Lambda)$ be a set
of simple-minded objects, \ie they satisfy the following conditions
\begin{itemize}
\item[(1)] $\Hom(X_i,\Sigma^m X_j)=0,~~\forall~m<0$,
\item[(2)] $\Hom(X_i,X_j)=\begin{cases} K & \text{if\ }i=j,\\
                                           0 & \text{otherwise},
                                           \end{cases}$
\item[(3)] $X_1,\ldots,X_r$ generates
$\mathcal{D}^{b}(\mod\Lambda)$.
\end{itemize}

On the graded algebra $\bigoplus_m \Hom(\bigoplus_i X_i,\Sigma^m
\bigoplus_i X_i)$ there is a strictly unital minimal
$A_\infty$-algebra structure. We will denote this $A_\infty$-algebra
by $\cx$. The conditions (1) and (2) imply that $\cx$ is positive,
while it follows from condition (3) that there is a triangle
equivalence
\[\cd(\cs)\longrightarrow \cd(\cx).\] This equivalence restricts to
triangle equivalences
\[\Psi:\per(\cs)\longrightarrow \per(\cx),\]
\[\Psi|:\cd_{fd}(\cs)\longrightarrow\cd_{fd}(\cx).\]
Thus we have the following commutative triangles of triangle
equivalences
\[\xymatrix{\cd^b(\mod\Lambda) \ar[r]^{\Phi}\ar[d]^{\Psi\circ\Phi} & \per(\cs)\ar[ld]^{\Psi}\\
\per(\cx)}
\xymatrix{\thick(I_1,\ldots,I_r) \ar[r]^{\Phi|}\ar[d]^{(\Psi\circ\Phi)|}& \cd_{fd}(\cs)=\thick(\bar{\cs})\ar[ld]^{\Psi|}\\
\cd_{fd}(\cx)=\thick(\bar{\cx}).}\]

Associated with $X_1,\ldots,X_r$ there is the decomposition
$1=id_{X_1}+\ldots+id_{X_r}$ of the identity of $\cx^0$ into the sum
of primitive orthogonal idempotents. Let $Y_1,\ldots,Y_r$ be
corresponding simple modules over $\cx$, and let $T_1,\ldots,T_r$ be
their images under a quasi-inverse of the equivalence
$(\Psi\circ\Phi)|$. Put $T=\bigoplus_i T_i$.

\begin{lemma}[Lemma~\ref{L:T} and Lemma~\ref{L:vanishingofposext}]\label{l:new-construction}
\begin{itemize}
\item[a)] $T$ generates $\thick(I_1,\ldots,I_r)$.
\item[b)] For $1\leq i,j\leq r$, and
$m\in\mathbb{Z}$,
\[\Hom(X_j,\Sigma^m T_i)=\begin{cases} K & \text{if }i=j\text{ and } m=0,\\
                                           0 & \text{otherwise}.
                                           \end{cases}\]

\item[c)] $T$ is isomorphic to
a bounded complex of finitely generated injectives.

\item[d)] Let $C$ be an object of
$\mathcal{D}^{-}(\mod \Lambda)$. If $\Hom(C,\Sigma^m T)=0$ for all
$m\in\mathbb{Z}$, then $C=0$.

\item[e)] $\Hom(T,\Sigma^m T)=0$ for $m>0$.
\end{itemize}
\end{lemma}
\begin{proof} a) b) e) hold because they hold after
applying the triangle equivalence $\Psi\circ\Phi$. c) is trivial. d)
follows from a).
\end{proof}

\begin{remark} From the appendix we see that, from the viewpoint of derived categories, finite-dimensional dg algebras (whose cohomology is) concentrated in non-positive degrees behave like ordinary finite-dimensional algebras.
The construction of $T$ and Lemma~\ref{l:new-construction} can be easily generalized to this more general setting, namely, the setting that $\Lambda$ is a finite-dimensional dg algebra (whose cohomology is)
concentrated in non-positive degrees. In the statement of d) one replaces
$\cd^{-}(\mod\Lambda)$ by the full subcategory of $\cd(\Lambda)$ of
dg $\Lambda$-modules $M$ such that $H^m(M)$ vanishes for
sufficiently large $m$ and each $H^m(M)$ ($m\in\mathbb{Z}$) is
finite-dimensional.
\end{remark}

\begin{remark}  The $A_\infty$-algebra $\cx$ can be computed as a
minimal model of the dg endomorphism algebra of the direct sum of
projective resolutions of $X_1,\ldots,X_r$. In fact, it is Koszul
dual to the dg algebra $\tilde{\Gamma}$ introduced in
Section~\ref{ss:t-structure}. Thus knowing that $\tilde{\Gamma}$ is
finite-dimensional a priori one can construct it from $\cx$ using
the dual bar construction, and vice versa. In particular, if the restriction of the $A_\infty$-structure of $\cx$ in degrees $0$, $1$ and $2$ is known, it is not hard to work out the precise structure of $\Gamma=H^0\tilde\Gamma$.  However, this does not
help us to understand when $\tilde{\Gamma}$ has cohomology
concentrated in degree $0$.
\end{remark}

\section{Appendix: Finite-dimensional non-positive dg algebras}\label{s:appendix}
Let $K$ be a field. Let $A$ be a finite-dimensional non-positive dg
$K$-algebra (associative with $1$), i.e. $A=\bigoplus_{i\leq 0} A^i$
with each $A^i$ finite-dimensional $K$-space and $A^i=0$ for $i\ll
0$.

Let $\cc(A)$ denote the category of (right) dg modules over $A$,
$\mathcal{D}(A)$ denote the derived category, $\mathcal{D}_{fd}(A)$
denote the finite-dimensional derived category, and $\per(A)$ denote
the perfect derived category.

 The 0-th cohomology
$\bar{A}=H^{0}(A)$ of $A$ is an ordinary $K$-algebra. Let $\Mod
\bar{A}$ and $\mod\bar{A}$ denote the category of (right) modules
over $\bar{A}$ and its subcategory consisting of those
finite-dimensional modules. Let $\pi:A\rightarrow \bar{A}$ be the
canonical projection. We view $\Mod \bar{A}$ as a subcategory of
$\cc(A)$ via $\pi$.

The total cohomology $H^{*}(A)$ of $A$ is a finite-dimensional
graded algebra with multiplication induced from the multiplication
of $A$. Let $M$ be a dg $A$-module. Then the total cohomology
$H^{*}(M)$ carries a graded $H^{*}(A)$-module structure, and hence a
graded $\bar{A}=H^{0}(A)$-module structure. In particular, a stalk
dg $A$-module concentrated in degree 0 is an $\bar{A}$-module.

\subsection{The standard $t$-structure}\label{s:standard-t-structure} We follow~\cite{Amiot09}
and~\cite{KellerYang10}, where the dg algebra is not necessarily
finite-dimensional.

Let $M=\ldots\rightarrow M^{i-1}\stackrel{d^{i-1}}{\rightarrow}
M^i\stackrel{d^i}{\rightarrow}M^{i+1}\rightarrow\ldots$ be a dg
$A$-module. We define the {\em truncation functors} $\tau_{\leq 0}$
and $\tau_{\geq 1}$ as follows:
\begin{eqnarray*}
\tau_{\leq 0}M&=&\ldots\rightarrow
M^{-2}\stackrel{d^{-2}}{\rightarrow}M^{-1}\stackrel{d^{-1}}{\rightarrow}\mathrm{ker}d^{0}\rightarrow
0\rightarrow\ldots\\
\tau_{\geq 1}M&=&\ldots\rightarrow 0\rightarrow
M^{1}/\mathrm{im}d^{0}\stackrel{d^{1}}{\rightarrow}M^2\stackrel{d^2}{\rightarrow}M^3\rightarrow\ldots
\end{eqnarray*}
Thanks to the assumption that $A$ is non-positive, $\tau_{\leq 0}M$
and $\tau_{\geq 1}M$ are again dg $A$-modules. Moreover we have a
distinguished triangle in $\mathcal{D}(A)$ \[\tau_{\leq
0}M\rightarrow M\rightarrow \tau_{\geq 1}M\rightarrow
\Sigma\tau_{\leq 0}M.\] These two functors define a $t$-structure
$(\mathcal{D}^{\leq 0},\mathcal{D}^{\geq 0})$ on $\mathcal{D}(A)$,
where $\mathcal{D}^{\leq 0}$ is the subcategory of $\mathcal{D}(A)$
consisting of dg $A$-modules with vanishing cohomology in positive
degrees, and $\mathcal{D}^{\leq 0}$ is the subcategory of
$\mathcal{D}(A)$ consisting of dg $A$-modules with vanishing
cohomology in negative degrees.

Immediately from the definition of the $t$-structure
$(\mathcal{D}^{\leq 0},\mathcal{D}^{\geq 0})$, we see that the heart
$\mathcal{H}=\mathcal{D}^{\leq 0}\cap\mathcal{D}^{\geq 0}$ consists
of those dg $A$-modules whose cohomology are concentrated in degree
0. Thus the functor $H^0$ induces an equivalence
\begin{eqnarray*}H^0:\mathcal{H}&\longrightarrow&\Mod\bar{A}.\\
M&\mapsto& H^{0}(M)
\end{eqnarray*}
The $t$-structure $(\mathcal{D}^{\leq 0},\mathcal{D}^{\geq 0})$ of
$\mathcal{D}(A)$ restricts to a $t$-structure of
$\mathcal{D}_{fd}(A)$ with heart equivalent to $\mod\bar{A}$.
 It is easy to see that
$\cd_{fd}(A)$ is generated by this heart, and hence generated by the
simple $\bar{A}$-modules.

\subsection{Morita reduction}
 Let $d$ be the differential of $A$. Then $d(A^0)=0$.

 Let $e$ be an idempotent of $A$. For degree reasons, $e$ must belong to
 $A^0$, and the graded
 subspace $eA$ of $A$ is a dg submodule: $d(ea)=d(e)a+ed(a)=ed(a)$.
 Therefore for each
 decomposition $1=e_1+\ldots+e_n$ of unity into the sum of primitive
 orthogonal idempotents, we have a direct sum decomposition
 $A=e_1A\oplus\ldots\oplus e_nA$ of $A$ into indecomposable dg
 $A$-modules.
 Moreover, if $e$ and $e'$ are two idempotents of $A$ such that
 $eA\cong e'A$ as ordinary modules over the ordinary algebra $A$,
 then this isomorphism is also an isomorphism of dg modules. Indeed,
 there are two elements of $A$ such that $fg=e$ and $gf=e'$. Again
 for degree reasons, $f$ and $g$ belong to $A^0$. So they induce
 isomorphisms of dg $A$-modules: $eA\rightarrow e'A$, $a\mapsto ga$
 and $e'A\rightarrow eA$, $a\mapsto fa$.  It follows that the above
 decomposition of $A$ into the direct sum of indecomposable dg modules is essentially unique. Namely, if
 $1=e'_1+\ldots+e'_n$ is another decomposition of the unity into the
 sum of primitive orthogonal idempotents, then $m=n$ and up to
 reordering, $e_1A\cong e'_1A$, $\ldots$, $e_nA\cong e'_nA$.

Let $A$ and $A'$ be two finite-dimensional non-positive dg algebras.
If $A$ and $A'$ are Morita equivalent as ordinary algebras, then
$\cc(A)$ and $\cc(A')$ are equivalent.

\subsection{The perfect derived category}

Since $A$ is finite-dimensional (thus has finite-dimensional total
cohomology), it follows that $\mathrm{per}(A)$ is a triangulated
subcategory of $\mathcal{D}_{fd}(A)$.

We assume, as we may, that $A$ is basic. Let $1=e_1+\ldots+e_n$ be a
decomposition of $1$ in $A$ into the sum of primitive orthogonal
idempotents. Since $d(x)=\lambda_1 e_{i_1}+\ldots+\lambda_s e_{i_s}$
implies that $d(e_{i_j}x)=\lambda_j e_{i_j}$, it follows that the
intersection of the space with basis $e_1,\ldots,e_n$ with the image
of the differential $d$ has a basis consisting of some $e_i$'s, say
$e_{r+1},\ldots,e_n$. It is easy to see that $e_{r+1} A,\ldots,e_n
A$ are homotopic to zero.

We say that a dg $A$-module $M$ is \emph{strictly perfect} if its
underlying graded module is of the form $\bigoplus_{j=1}^{N}R_j$,
where each $R_j$ is isomorphic to a shifted copy of some $e_i A$
($1\leq i\leq n$), and if its differential is of the form
$d_{int}+\delta$, where $d_{int}$ is the direct sum of the
differential of the $R_j$'s, and $\delta$, as a degree $1$ map from
$\bigoplus_{j=1}^{N}R_j$ to itself, is a strictly upper triangular
matrix whose entries are in $A$. It is \emph{minimal perfect} if in
addition no $R_j$ is isomorphic to any shifted copy of
$e_{r+1}A,\ldots, e_n A$, and the entries of $\delta$ are in the
radical of $A$, \confer~\cite{Plamondon10}.

\begin{lemma}\label{l:minimal-perfect-resolution} Let $M$ be a
dg $A$-module belonging to $\per(A)$. Then $M$ is quasi-isomorphic
to a minimal perfect dg $A$-module.
\end{lemma}
\begin{proof} Bearing in mind that $e_1 A,\ldots, e_r A$ have local
 endomorphism algebras and $e_{r+1}A,\ldots, e_n A$ are homotopic to zero,
 we prove the assertion as in~\cite{Plamondon10}. \end{proof}

\subsection{Simple modules}\label{S:simpledg} We assume
that $A$ is basic and that $K$ is algebraically closed.

According to the preceding subsection, we may assume that there is a
decomposition $1=e_1+\ldots+e_r+e_{r+1}+\ldots+e_n$ of the unity of
$A$ into a sum of primitive orthogonal idempotents such that
$1=\bar{e}_1+\ldots+\bar{e}_r$ is a decomposition of $1$ in
$\bar{A}$ into a sum of primitive orthogonal idempotents.

Let $S_1,\ldots,S_r$ be a complete set of representatives of
isomorphism classes of simple $\bar{A}$-modules. Then
\[\mathcal{H}om_{A}(e_i A,S_{j})=\begin{cases} K & \text{if }i=j,\\
                                           0 & \text{otherwise}.
                                           \end{cases}\]
Here for two dg $A$-modules $M$ and $N$,  $\cHom_A(M,N)$ denotes the
complex whose degree $p$ component of consists of those $A$-linear
maps from $M$ to $N$ which are homogeneous of degree $p$, and whose
differential  takes a homogeneous map $f$ of degree $p$ to $d_N\circ
f-(-1)^p f\circ d_M$. Therefore we have
\[\Hom_{\cd(A)}(e_i A,\Sigma^m S_j)=\begin{cases} K & \text{ if } i=j \text{ and } m=0,\\ 0 & \text{ otherwise}.\end{cases}\]
Moreover, $\{e_1 A,\ldots,e_rA\}$ and $\{S_1,\ldots,S_r\}$
characterize each other by this property. On the one hand, if $M$ is
a dg $A$-module such that for some integer $1\leq j\leq r$
\[\Hom_{\cd(A)}(e_i A,\Sigma^m M)=\begin{cases} K & \text{ if } i=j \text{ and } m=0,\\ 0 & \text{ otherwise},\end{cases}\]
 then $M$ is isomorphic in $\cd(A)$ to $S_j$.
On the other hand, let $M$ be an object of $\per(A)$ such that for
some integer $1\leq i\leq r$
\[\Hom_{\cd(A)}(M,\Sigma^m S_j)=\begin{cases} K & \text{ if } i=j \text{ and } m=0,\\ 0 & \text{ otherwise}.\end{cases}\]
Then by replacing $M$ by its minimal perfect resolution
(Lemma~\ref{l:minimal-perfect-resolution}), we see that $M$ is
isomorphic in $\cd(A)$ to $e_i A$.

Further, recall from Section~\ref{s:standard-t-structure} that $\cd_{fd}(A)$ admits a standard $t$-structure whose heart is equivalent to $\mod\bar{A}$. This implies that the simple modules $S_1,\ldots,S_r$ form a set of simple-minded objects in $\cd_{fd}(A)$.

\subsection{The Nakayama functor}

For a complex $M$ of $K$-vector spaces, we define its dual as
$\mathrm{D}(M)=\mathcal{H}om_{K}(M,K)$, where the last $K$ is
considered as a complex concentrated in degree 0. One checks that
$\mathrm{D}$ defines a duality between finite-dimensional dg
$A$-modules and finite-dimensional $A^{op}$-modules..

Let $e$ be an idempotent of $A$ and $M$ a dg $A$-module. Then we
have a canonical isomorphism
\[\mathcal{H}om_{A}(eA,M)\cong Me.\]
If in addition each component of $M$ is finite-dimensional, we have
canonical isomorphisms
\[\mathcal{H}om_{A}(eA,M)\cong Me\cong \mathrm{D}\mathcal{H}om_{A}(M,\mathrm{D}(Ae)).\]

We define the Nakayama functor $\nu:\cc(A)\rightarrow\cc(A)$ by
$\nu(M)=\mathrm{D}\mathcal{H}om_{A}(M,A)$ \cite[Section
10]{Keller94}. We have canonical isomorphisms \[D\cHom_A(M,N)\cong
\cHom_A(N,\nu M)\] for strictly perfect dg $A$-module $M$ and any dg
$A$-module $N$. We have $\nu(eA)=\mathrm{D}(Ae)$ for an idempotent
$e$ of $A$, and the functor $\nu$ induces a triangle equivalences
between the subcategories $\per(A)$ and $\thick(\mathrm{D}(A))$ of
$\cc(A)$ with quasi-inverse given by
$\nu^{-1}(M)=\mathcal{H}om_{A}(\mathrm{D}(A),M)$.

 Let
$e_1,\ldots,e_r$ and $S_1,\ldots,S_r$ be as in the preceding
subsection. Then we have
\[\mathcal{H}om_{A}(S_j, \mathrm{D}(Ae_i))\cong D\cHom_A(e_i A, S_j)=\begin{cases} K & \text{if }i=j,\\
                                           0 & \text{otherwise}.
                                           \end{cases}\]
That is, \[\Hom_{\cd(A)}(S_j,\Sigma^m
\mathrm{D}(Ae_i))=\begin{cases} K & \text{ if } i=j \text{ and }
m=0,\\ 0 & \text{ otherwise}.\end{cases}\] Moreover,
$\{\mathrm{D}(Ae_1),\ldots,\mathrm{D}(Ae_r)\}$ and
$\{S_1,\ldots,S_r\}$ characterize each other in $\mathcal{D}(A)$ by
this property. This follows from the arguments in the preceding
subsection by applying the functors $\nu$ and $\nu^{-1}$.

\end{document}